\documentclass[12pt]{amsart}
\usepackage{stmaryrd}
\usepackage{mathrsfs}
\usepackage{amssymb}
\usepackage{amsmath}
\usepackage{titletoc}
\usepackage{extarrows}
\usepackage{varioref}
\usepackage{bm}
\usepackage{xypic}
\usepackage[normalem]{ulem}
\usepackage[mathscr]{eucal}
\usepackage{tikz}
\usepackage{tikz-cd}
\usepackage[all]{xy}
\usepackage{color}
\usepackage{soul}
\usepackage{amscd}
\usepackage{amsfonts}
\usepackage{verbatim}
\usepackage[colorlinks,linkcolor=blue,citecolor=blue, pdfstartview=FitH]{hyperref}
\usepackage{graphicx}

\newcommand\ord{\operatorname{ord}}

\newcommand\Tr{\operatorname{Tr}}
\newcommand\Vol{\operatorname{Vol}}

\newcommand\card{\operatorname{card}}

\numberwithin{equation}{section}
\newtheorem{Definition}{Definition}[section]
\newtheorem{Theorem}[Definition]{Theorem}
\newtheorem{Lemma}[Definition]{Lemma}
\newtheorem{Proposition}[Definition]{Proposition}

\begin{document}

\title{Robba's method on Exponential sums}
\author{Peigen Li}
\address{Yau Mathematical Sciences Center, Tsinghua University, Beijing 100084, P. R. China}
\email{lpg16@mails.tsinghua.edu.cn}
\begin{abstract}
     In this article, we use Robba's method to give an estimate of the Newton polygon for the $L$-function and we can draw the Newton polygon in some special cases.
\end{abstract}
\maketitle


\section{Introduction}
The basic objects of this study are exponential sums on a torus of dimension $n$ defined over a finite field $k$ with $\textrm{char}(k)=p$. Our methods are based on the work of Dwork, Adolphson and Sperber. In \cite{robba}, Robba gives an explicit calculation of one variable twisted exponential sums. In fact, his method can be applied to the case of multi-variables. 

Let $\zeta_p$ be a primitive $p$-th root of unity. Let $\psi$ be the additive character of $k$ given by $\psi(t)=\zeta_p^{\Tr_{k/\mathbf{F}_p}(t)}$. Let $f$ be a Laurent polynomial and write 
$$
f=\sum_{i=1}^Na_ix^{w_i}\in k[x_1,\cdots,x_n,x_1^{-1},\cdots,x_n^{-1}].
$$ We assume that $a_i\neq 0$ for all $i$. Define exponential sums 
$$
S_i(f)=\sum_{x\in \mathbf{T}^n(k_i)}\psi(\Tr_{k_i/k}(f(x))),
$$ where $k_i$ are the extensions of $k$ of degree $i$. The $L$-function is defined by 
$$
L(f,t)=\exp\Big(\sum_{i=1}^{\infty}S_i(f)t^i/i\Big).
$$

In \cite[section 2]{Ado1989}, Adolphson and Sperber use 
Dwork's method to prove that $L(f,t)^{(-1)^{n-1}}$ is a polynomial when $f$ is nondegenerate. Moreover, they give a low bound of the Newton polygon of $L(f,t)^{(-1)^{n-1}}$ in \cite[section 3]{Ado1989}, which we call Hodge polygon in this article. In our study, we want to give a more precise result about the Newton polygon when $f$ has only $n$ terms, that is $N=n$. Note that if we assume that $J=(w_1,\cdots,w_n)$ is invertible in $\mathbf{M}_n(\mathbf{R})$, we can found a solution $b=(b_1,\cdots,b_n)\in \bar{k}^{\times}$ such that $a_ib^{w_i}=1$ for all $i$. From now on, we assume that $(p,\det J)=1$, $k=\mathbf{F}_p$ and 
$$
f=\sum_{i=1}^nx^{w_i}.
$$

Let $\Delta(f)$ be the Newton polyhedron at $\infty$ of $f$ which is defined to be the convex hull in $\mathbf{R}^n$ of the set $\left\{w_j\right\}_{j=1}^n\cup \left\{(0,\cdots,0)\right\}$ and let $C(f)$ be the convex cone generated by $\left\{w_j\right\}_{j=1}^n$ in $\mathbf{R}^n$. Let $\Vol(\Delta(f))$ be the volume of $\Delta(f)$ with respect to Lebesgue measure on $\mathbf{R}^n$. 
We say $f$ is \textit{nondegenerate with respect to $\Delta(f)$} if for any face $\sigma$ of $\Delta(f)$ not containing the origin, the Laurent polynomials $\frac{\partial f_{\sigma}}{\partial x_i}$, $i=1,\cdots,n$ 
have no common zero in $(\bar{k}^{\times})^n$, where $f_{\sigma}=\sum_{w_j\in \sigma}a_jx^{w_j}$. Set $M(f)=C(f)\cap \mathbf{Z}^n$. Note that $(p,\det J)=1$ implies that $f$ is nondegenerate. Since we have assumed that $J$ is invertible, any element $u\in M(f)$ can be uniquely written 
\begin{equation}\label{express}
u=\sum_{i=1}^nr_iw_i.
\end{equation}We define a weight on $M(f)$
$$
w(u):=\sum_{i=1}^nr_i.
$$
Note that the set of all elements $u\in M(f)$ such that all $0\leq r_i<1$ 
in the expression (\ref{express}) form a fundamental domain of the lattice $M(f)$. Denote it by $S(\Delta)$. Note that $\card(S(\Delta))=n!\Vol{\Delta(f)}=\det(J)$ and $(p,\det J)=1$ imply that 
$S(\Delta)$ has a natural $p$-action. For any $u=r_1w_1+\cdots+r_nw_n\in S(\Delta)$, define 
$$
p.u=\sum_{i=1}^n\left\{pr_i\right\}w_i,
$$ where $\left\{pr_i\right\}$ is the fractional part of $pr_i$ for each $i$. We say $S(\Delta)$ is \textit{$p$-stable under weight function} if $w(u)=w(p.u)$ for any $u\in S(\Delta)$. Now we give our main result.
\begin{Theorem}
    Suppose that $f=x^{w_1}+\cdots+x^{w_n}$ with $w_i\in\mathbf{Z}^n$ and $(p,\det J)=1$. The Newton polygon of $L(f,t)^{(-1)^{n-1}}$ coincides with the Hodge polygon of $\Delta(f)$ if and only if $S(\Delta)$ is $p$-stable under weight function.
\end{Theorem}
Wan uses the Gauss sum to give an explicit formula of the $L$-function for the diagonal Laurent polynomial. Then he uses Stickelberger's theorem to give a proof of above theorem. See \cite[Theorem 3.4]{Wan}. In this article, we use Robba's method to prove above theorem. Indeed, Robba's method can also be applied to prove \cite[Theorem 3.10]{Ado1989} and it is easier than the method used in \cite[\S 3]{Ado1989}.
\section{p-adic estimates}
Let $\mathbf{Q}_p$ be the $p$-adic numbers. Let $\Omega$ be the completion of the algebraic closure of $\mathbf{Q}_p$. Denote by ``ord" the additive valuation on $\Omega$ normalized by $\ord(p)=1$. The norm on $\Omega$ is given by $|u|=p^{-\ord(u)}$ for any $u\in \Omega$. 

Note that there is an integer $M$ such that $w(M(f))\subset \frac{1}{M}\mathbf{Z}$. In \cite[section 1]{Ado1989},  Adolphson and Sperber introduce a filtration on $R(f):=k[x^{M(f)}]$ given by
\begin{displaymath}
R(f)_{i/M}=\left\{\sum_{u\in M(f)}b_ux^u|w(u)\leq i/M~\textrm{for all $u$ with}~b_u\neq 0\right\}.
\end{displaymath}
The associated graded ring is
\begin{displaymath}
\bar{R}=\bigoplus_{i\in \mathbf{Z}_{\geq 0}}\bar{R}^{i/M},
\end{displaymath}
where
\begin{displaymath}
\bar{R}^{i/M}=R(f)_{i/M}/R(f)_{(i-1)/M}.
\end{displaymath}
For $1\leq i\leq n$, let $\bar{f}_i$ be the image of $x_i\frac{\partial f}{\partial x_i}\in R(f)_1$ in $\in \bar{R}^{1}$. Let $\bar{I}$ be the ideal generated by $\bar{f}_1,\dots,\bar{f}_n$ in $\bar{R}$. By \cite[Theorem 2.14]{Ado1989} and \cite[Theorem 2.18]{Ado1989}, $\bar{f}_1,\dots,\bar{f}_n$ in $\bar{R}$ form a regular sequence in $\bar{R}$ and $\dim_k\bar{R}/\bar{I}=n!\Vol(\Delta(f))$. For each integer $i$, we have a decomposition
\begin{equation}\label{Decompo}
\bar{R}^{i/M}=\bar{V}^{i/M}\oplus (\bar{R}^{i/M}\cap \bar{I}).
\end{equation}
Set $a_i$=$\dim_{k}\bar{V}^{i/M}$. 

For a non-negative integer $l$, set
\begin{displaymath}
W(l)=\card\left\{u\in M(f)|w(u)=\frac{l}{M}\right\}.
\end{displaymath}
Note that this is a finite number for each $l$. Set
\begin{displaymath}
H(i)=\sum_{l=0}^n(-1)^l\binom{n}{l}W(i-l M).
\end{displaymath}
\begin{Lemma}\label{Hodge-index}
With the notation above. Suppose that $f$ is nondegenerate. Then $H(i)=a_i$ for all integer $i\geq 0$. Moreover, we have
\begin{displaymath}
H(k)=0~\mathrm{for} ~k>nM,\quad
\sum_{k=0}^{nM}H(k)=n!\Vol(\Delta(f)).
\end{displaymath}
\end{Lemma}
\begin{proof}
By \cite[Theorem 2.14]{Ado1989}, $\left\{\bar{f}_i\right\}_{i=1}^n$ form a regular sequence in $\bar{R}$. So
$$P_{\bar{R}/\bar{I}}(t)=P_{\bar{R}}(t)(1-t^M)^n,$$
where $P_{\bar{R}/\bar{I}}$ (resp. $P_{\bar{R}}$) is the Poincar\'{e} series of $\bar{R}/\bar{I}$ (resp. $\bar{R}$). On the other hand, we have
\begin{displaymath}
P_{\bar{R}/(\bar{f}_1,\dots,\bar{f}_n)}=\sum_{i=0}^{\infty}a_it^i,~P_{\bar{R}}(t)=\sum_{i=0}^{\infty}W(i)t^i.
\end{displaymath}
Hence
\begin{displaymath}
a_i=\sum_{l=0}^n(-1)^l\binom{n}{l}W(i-l M)=H(i).
\end{displaymath}
The second assertion follows from \cite[Lemma 2.9]{kouchnirenko1976polyedres}.
\end{proof}
Note that $\bar{R}/\bar{I}$ has a finite basis $S=\left\{x^{u}|u\in S(\Delta)\right\}$ and $\card(S)=n!\Vol(\Delta(f))$.
\begin{Definition}
The Hodge polygon $HP(\Delta)$ of $\Delta(f)$ is defined to be the convex polygon in $\mathbf{R}^2$ with vertices $(0,0)$ and
\begin{displaymath}
\Big(\sum_{k=0}^mH(k),\frac{1}{M}\sum_{k=0}^mkH(k)\Big).
\end{displaymath}
\end{Definition}

Consider the Artin-Hasse exponential series: 
$
E(t)=\exp\Big(\sum_{i=0}^{\infty}\frac{t^{p^i}}{p^i}\Big).
$
By \cite[Lemma 4.1]{dwork1}, the series $\sum_{i=0}^{\infty}\frac{t^{p^i}}{p^i}$ has a zero at $\gamma\in \Omega$ such that $\ord\gamma=1/(p-1)$ and $\zeta_p\equiv 1+\gamma\mod\gamma^2$. Set
$$
\theta(t)=E(\gamma t)=\sum_{i=0}^{\infty}c_i t^i.
$$ The series $\theta(t)$ is a splitting function in Dwork's terminology \cite[\S 4a]{dwork1}. In particular, we have $\ord c_i\geq i/(p-1)$, $\theta(t)\in \mathbf{Q}_p(\zeta_p)[[t]]$ and $\theta(1)=\zeta_p$.
Fix an $M$-th root $\widetilde{\gamma}$ of $\gamma$ in $\Omega$. Let $K=\mathbf{Q}_p(\widetilde{\gamma})$, and $\mathcal{O}_K$ the ring of integers of $K$. Let $\hat{a}_j\in K$ be the Techm\"uller lifting of $a_j$ and set 
$$\hat{f}(x)=\sum_{j=1}^N\hat{a}_jx^{\omega_j}\in K[x_1,x_1^{-1},\cdots,x_n,x_n^{-1}].$$
Consider the following spaces of $p$-adic functions
$$
B_0=\left\{\sum_{u\in M(f)}A_u\widetilde{\gamma}^{Mw(u)}x^u|A_u\in \mathcal{O}_K,A_u\rightarrow 0 ~\textrm{as} ~u\rightarrow 0\right\},
$$
$$
B=\left\{\sum_{u\in M(f)}A_u\widetilde{\gamma}^{Mw(u)}x^u|A_u\in K,A_u\rightarrow 0 ~\textrm{as} ~u\rightarrow 0\right\}.
$$
Set $\gamma_l=\sum\limits_{i=0}^{l}\gamma^{p^i}/p^i,h(t)=\sum\limits_{l=0}^{\infty}\gamma_lt^{p^l}$. 
Define 
$$
H(x)=\sum_{j=1}^nh(x^{w_j}),~ F_0(x)=\prod_{i=1}^n\theta(x^{w_i})=\sum_{v\in M(f)}h_vx^v.
$$ Define an operator $\psi$ on formal Laurent series by
\begin{displaymath}
\psi(\sum_{u\in \mathbf{Z}^n}a_ux^u)=\sum_{u\in\mathbf{Z}^n}a_{pu}x^u.
\end{displaymath}
Let $\alpha=\psi \circ F_0(x)$. For $i=1,\cdots,n$, define operators 
$$
E_i=x_i\partial/\partial x_i,~\hat{D}_i=E_i+E_i(H)
$$
By \cite[Corollary 2.9]{Ado1989}, we have 
\begin{eqnarray*}
L(f,t)^{(-1)^{n-1}}=\det(1-t\alpha|B/\sum_{i=1}^n\hat{D}_iB).
\end{eqnarray*}
By \cite[Therorem 2.18, Theorem A.1]{Ado1989}, $S=\left\{x^u\right\}_{u\in S(\Delta)}$ is a free basis of $B/\sum_{i=1}^n\hat{D}_iB$. For any $u\in M(f), u'\in S(\Delta)$, define $A(u,u')$ by the relations
\begin{displaymath}
x^{u}\equiv\sum_{u'\in S(\Delta)}A(u,u')x^{u'}\mod\sum_{i=1}^n\hat{D}_iB.
\end{displaymath}
For any $u, u'\in S(\Delta)$, define $\gamma(u,u')$ by the relations
\begin{displaymath}
\alpha(x^{u})\equiv\sum_{u'\in S(\Delta)}\gamma(u,u')x^{u'}\mod\sum_{i=1}^n\hat{D}_iB.
\end{displaymath}
The main purpose is to give estimate for the $p$-adic valuations of the coefficients $\gamma(u,u')$. 

For any $u\in M(f)$, there is a unique $u'\in S(\Delta)$ such that 
$$
u\in S_{u'}=\left\{u'+\sum_{i=1}^n\mathbf{Z}_{\geq 0}w_i\right\}. 
$$ Set $R_{u'}=\left\{\xi=\sum a_ux^u\in B_0|u\in S_{u'}\right\}$.
\begin{Lemma}\label{est-lemma2}
For any $u\in M(f)$, we have $A(u,u')=0$ if $u\notin S_{u'}$, $\ord(A(u,u'))\geq \frac{w(u')-w(u)}{p-1}$ if $u\in S_{u'}$.
\end{Lemma}
\begin{proof}
     The first assertion follows from the facts that 
     $$
B_0=\bigoplus_{u'\in S(\Delta)}R_{u'}
     $$ and $\hat{D}_i(R_{u'})\subset R_{u'}$ for any $i$ and $u'$. Suppose that $u\in S_{u'}$. By \cite[Proposition 3.1]{Ado1989}, there exit $A\in \mathcal{O}_K$ and $\xi_1,\cdots,\xi_n\in B_0$ such that 
     $$
     \widetilde{\gamma}^{Mw(u)}x^u=A\widetilde{\gamma}^{Mw(u')}x^{u'}+\sum_{i=1}^n\hat{D}_i\xi_i.
     $$Hence, we have 
     $$
\ord(A(u,u'))=\ord(A\widetilde{\gamma}^{Mw(u')-Mw(u)})\geq \frac{w(u')-w(u)}{p-1}.
     $$
\end{proof}


\begin{Proposition}\label{est-diag-basic}
For any $u,u'\in S(\Delta)$, we have
$$\ord(\gamma(u,u'))=\left\{\begin{array}{cc}
+\infty&\hbox{if } p.u'-u\neq 0,\\
\frac{pw(u')-w(u)}{p-1}&\hbox{if }p.u'-u=0.
\end{array}
\right.$$ $\ord(\gamma(u,u'))=+\infty$ means that $\gamma(u,u')=0$.
\end{Proposition}

\begin{proof}
     Note that 
     \begin{eqnarray*}
          \alpha(x^u)&=&\psi(x^uF_0(x))=\sum_{v\in M(f)}h_{pv-u}x^v
          \\&\equiv&\sum_{u'\in S(\Delta)}\sum_{v\in M(f)}h_{pv-u}A(v,u')x^{u'}\mod\sum_{i=1}^n\hat{D}_i B.
     \end{eqnarray*}
     By Lemma \ref{est-lemma2}, $A(v,u')=0$ when $v\notin S_{u'}$. Hence, we have 
\begin{equation}\label{est-diag}
\gamma(u,u')=h_{pu'-u}+\sum_{v\in M(f)-S(\Delta)}h_{pv-u}A(v,u').
\end{equation}
Assume that $v=u'+\sum\limits_{i=1}^ns_iw_i$ with $s_i\in \mathbf{Z}_{\geq 0}$. Note that
\begin{displaymath}
h_{pv-u}=\prod_{j=1}^nc_{k_j},
\end{displaymath}
where $(k_1,\dots,k_n)\in \mathbf{Z}_{\geq 0}^n$ satisfies the equation
\begin{equation}\label{one2}
\sum_{i=1}^nk_i w_i=pv-u=pu'-u+p\sum_{i=1}^ns_iw_i.
\end{equation}
If $p.u'-u\neq 0$, the above equation has no integer solution which implies that $\gamma(u,u')=0$. If $p.u'-u=0$, suppose that $pu'-u=r_1w_1+\cdots+r_nw_n$ with $r_i\in\mathbf{Z}_{\geq 0}$ for all $i$. Note that $r_i\leq p-1$ for all $i$ and $w(pu'-u)=pw(u')-w(u)=r_1+\cdots+r_n$.
By (\ref{one2}), we have $k_i=r_i+ps_i$ for each $i$. Hence, by Lemma \ref{est-lemma2} and the estimate $\ord(c_i)\geq \frac{i}{p-1}$, we have
\begin{eqnarray*}
\ord(h_{pv-u}A(v,u'))\geq \sum_{i=1}^n\frac{k_i-s_i}{p-1}
=\sum_{i=1}^ns_i+\frac{pw(u')-w(u)}{p-1}.
\end{eqnarray*}
If $v\notin S(\Delta)$, there is some $i$ such that $s_i>0$, we have 
$$
\ord(h_{pv-u}A(v,u'))>\frac{pw(u')-w(u)}{p-1}.
$$
If $v=u'\in S(\Delta)$, we have $k_i=r_i\leq p-1$ for all $i$. Note that 
$$
\theta(t)\equiv \exp(\gamma t)\mod t^p.
$$ We have $\ord(c_{i})=\ord(\frac{\gamma^i}{i!})=\frac{i}{p-1}$ for any $i\leq p-1$. Hence 
\begin{displaymath}
\ord(h_{pu'-u})=\sum_{i=1}^n\ord(c_{r_i})=\frac{1}{p-1}\sum_{i=1}^nr_i=\frac{pw(u')-w(u)}{p-1}.
\end{displaymath}
By (\ref{est-diag}), we have
\begin{displaymath}
\ord(\gamma(u,u'))=\ord(h_{pu'-u})=\frac{pw(u')-w(u)}{p-1}.
\end{displaymath}
\end{proof}

\begin{Theorem}\label{Newton-polygon-diag}
Suppose that $f=\sum_{j=1}^nx^{w_j}$ and $(p,\det J)=1$. The Newton polygon of $L(\mathbf{T}^n,f,t)^{(-1)^{n-1}}$ coincides with the Hodge polygon $HP(\Delta)$ if and only if $S(\Delta)$ is $p$-stable under weight function.
\end{Theorem}
\begin{proof}
By \cite[Corollary 3.11]{Ado1989}, the Newton polygon of $L(\mathbf{T}^n,f,t)^{(-1)^{n-1}}$ lies above the Hodge polygon of $HP(\Delta)$ with same endpoints and the matrix $\Gamma:=(\gamma(u,u'))_{u,u'\in S(\Delta)}$ is invertible. By Proposition \ref{est-diag-basic}, $\gamma(u,u')\neq 0$ if and only if $p.u'-u=0$. 
Hence there is
exactly one non zero element in every column and row of $\Gamma$. Let $S(d,u)$ be the orbit of $u$ under the $p$-action with exactly $d$ elements. Suppose that $S(d,u)=\left\{u_1,\cdots,u_d\right\}$, where $u_i=p^{i-1}.u$. By Proposition \ref{est-diag-basic}, we have
\[\alpha(x^{u_1},\cdots,x^{u_d})=(x^{u_1},\cdots,x^{u_d})
\left(\begin{array}{ccc}
0 & \gamma_{21} & \dots \\
0 & 0 & \dots \\
\vdots & \vdots & \gamma_{dd-1} \\
\gamma_{1d} & \cdots & 0 \\
\end{array}\right)\]
where $\gamma_{ij}=\gamma(u_i,u_j)$. Thus
$$
\det(1-\alpha t)=\prod_{S(d,u)}(1-t^d\lambda_{u}),
$$
where the above product runs through all the obits of $S(\Delta)$ under the $p$-action and $\lambda_u=\gamma_{1d}\gamma_{21}\cdots\gamma_{dd-1}$ with
\begin{eqnarray*}
\ord(\lambda_u)&=&\ord(\gamma_{1d}\gamma_{21}\cdots\gamma_{dd-1})
\\&=&\frac{pw(u_d)-w(u_1)}{p-1}+\cdots+\frac{pw(u_{d-1})-w(u_d)}{p-1}
\\&=&\sum_{i=0}^{d-1}w(p^i.u).
\end{eqnarray*}
Set $f_{u,d}=1-t^d\lambda_{u'}$ and 
\begin{displaymath}
g_{u,d}=\prod_{i=0}^{d-1}(1-tp^{w(p^i.u)}).
\end{displaymath}
Note that the Newton polygon of $f_{u,d}$ always lies above the Newton polygon of $g_{u,d}$ and the Newton polygon of the polynomial $\prod_{S(d,u)}g_{u,d}$ is $HP(\Delta)$. Hence $HP(\Delta)$ coincides with the Newton polygon of $\det(1-\alpha t)$ if and only if the Newton polygons of $g_{u,d}$ and $f_{u,d}$ coincide for each $u$.

When $S(\Delta)$ is $p$-stable under weight function. We have $w(u)=w(p.u)=\cdots=w(p^{d-1}.u)$ for each $u$. Hence, the Newton polygons of $g_{u,d}$ and $f_{u,d}$ coincides for each $u$.

Conversely, if the Newton polygons of $g_{u,d}$ and $f_{u,d}$ coincide for each $u$. Since both polygons have same endpoints, we have $w(u)=w(p.u)=\cdots=w(p^{d-1}.u)$ for each $u$. Hence $S(\Delta)$ is $p$-stable under weight function.
\end{proof}

\bibliographystyle{plain}
\bibliography{ref/ref}

\end{document}